\documentclass{amsart}

\usepackage{amsmath}
\usepackage{amsthm}
\usepackage{amsfonts}
\usepackage{amssymb}

\newcommand\R{{\mathbb{R}}}
\newcommand\C{{\mathbb{C}}}

\renewcommand\P{{\mathbf{P}}}
\newcommand\E{{\mathbf{E}}}

\newcommand\eps{{\varepsilon}}

\newcommand\tr{\operatorname{trace}}

\theoremstyle{plain}
  \newtheorem{theorem}{Theorem}

  \newtheorem{proposition}[theorem]{Proposition}

\theoremstyle{definition}
  \newtheorem{definition}[theorem]{Definition}
  
  \newtheorem{remark}[theorem]{Remark}

\include{psfig}

\begin{document}

\title[Bulk universality for Wigner matrices]{Bulk universality
 for Wigner hermitian matrices with subexponential decay}

\author[L. Erd\H os]{L\'aszl\'o Erd\H{o}s}
\address{Institute of Mathematics, University of Munich, Theresienstr.
 39, D-80333 Munich, Germany} 
\email{lerdos@math.lmu.de}
\thanks{L. Erd\H{o}s is partially supported by SFB-TR 12 Grant of the 
German Research Council.}

\author[J. Ram\'irez]{Jos\'e Ram\'irez}
\address{Department of Mathematics, Universidad de Costa Rica
San Jose 2060, Costa Rica} 
\email{ALEXANDER.RAMIREZGONZALEZ@ucr.ac.cr}
\thanks{}

\author[B. Schlein]{Benjamin Schlein}
\address{Department of Pure Mathematics and Mathematical Statistics,
University of Cambridge, Wilberforce Rd, Cambridge CB3 0WB, UK} 
\email{b.schlein@dpmms.cam.ac.uk}
\thanks{}

\author[T. Tao]{Terence Tao}
\address{Department of Mathematics, UCLA, Los Angeles CA 90095-1555}
\email{tao@math.ucla.edu}
\thanks{T. Tao is is supported by a grant from the MacArthur Foundation,
 by NSF grant DMS-0649473, and by the NSF Waterman award.}

\author[V. Vu]{Van Vu}
\address{Department of Mathematics, Rutgers, Piscataway, NJ 08854}
\email{vanvu@math.rutgers.edu}
\thanks{V. Vu is supported by NSF grant  DMS-0901216 and DOD grant  AFOSAR-FA-9550-09-1-0167}

\author[H.-T. Yau]{Horng-Tzer Yau}
\address{Department of Mathematics, Harvard University
Cambridge MA 02138, USA4} 
\email{htyau@math.harvard.edu}
\thanks{H.-T. Yau is partially supported by NSF grants DMS-0757425, DMS-0804279}

\begin{abstract} We consider the ensemble of
$n \times n$ Wigner hermitian matrices
 $H = (h_{\ell k})_{1 \leq \ell,k \leq n}$ that
 generalize the Gaussian unitary ensemble (GUE).
 The matrix elements 
$h_{k\ell} =  \bar  h_{ \ell k}$ are given by
 $h_{\ell k} = n^{-1/2}
 ( x_{\ell k} + \sqrt{-1} y_{\ell k} )$, where $x_{\ell k}, y_{\ell k}$ for
 $1 \leq \ell < k \leq n$ are i.i.d. random variables with mean zero and variance 
$1/2$, $y_{\ell\ell}=0$ and $x_{\ell \ell}$
have mean zero and variance $1$.
 We assume the distribution of
$x_{\ell k}, y_{\ell k}$ to have subexponential decay.  
In \cite{ERSY2}, four of the authors recently established that the gap distribution 
and averaged $k$-point correlation of these matrices were \emph{universal}
 (and in particular, agreed with those for GUE) assuming additional 
regularity hypotheses on the  $x_{\ell k}, y_{\ell k}$.
In \cite{TVbulk}, the other two authors, using a different method, established
 the same conclusion assuming instead some moment and support conditions 
on the $x_{\ell k}, y_{\ell k}$.  In this short note we observe that the 
arguments  of \cite{ERSY2} and \cite{TVbulk} can be combined to establish 
universality of the gap distribution and averaged $k$-point correlations 
for all Wigner matrices (with subexponentially decaying entries), with no 
extra assumptions.
\end{abstract}

\maketitle

\section{Introduction}



This paper is concerned with the bulk eigenvalue statistics of large 
Hermitian Wigner random matrices, which we now pause to define.

\begin{definition}[Wigner matrices]\label{def:Wignermatrix} 
Let $n$ be a large number. A \emph{Wigner Hermitian matrix} (of size $n$)
 is defined as  a random Hermitian $n \times n$ matrix $H=H_n$ with upper
triangular complex entries $h_{\ell k} = n^{-1/2} 
(x_{\ell k} + \sqrt{-1} y_{\ell k})$ for $1 \leq \ell < k \leq n$ 
and diagonal real entries $h_{\ell \ell} = n^{-1/2} x_{\ell \ell}$ 
for $1 \leq \ell \leq n$, where

\begin{itemize}
\item For $1 \leq \ell < k \leq n$, the $x_{\ell k}, y_{\ell k}$ are 
i.i.d. copies of a real random variable $x$ with  mean zero and variance $1/2$.

\item For $1 \leq \ell \leq n$, the $x_{\ell \ell}$ are i.i.d. copies of 
a real random variable $\tilde x$
with mean zero and  variance $1$.

\item $x, \tilde x$ have subexponential decay, i.e., there are 
constants $C, C'$ such that 
$\P( |x| \ge t^C) \le \exp(- t), \;  \P( |\tilde x| \ge t^C) \le \exp(- t)$, 
for all $t \ge C'$.
\end{itemize} \end{definition}

We refer to $x$ as the \emph{atom distribution} of $H$.
 An important special case of a Wigner Hermitian
 matrix is the \emph{Gaussian unitary ensemble} (GUE), in which $x, \tilde x$ 
are Gaussian random variables with mean zero and variance $1/2$, $1$ respectively.

Let $\lambda_1 \leq \ldots \leq \lambda_n$ be the eigenvalues of a Wigner Hermitian
 matrix $H_n$, and let $p_n(x_1,\ldots,x_n)$ be the associated (symmetric) density function, normalized to have total mass $1$.  For any $1 \leq k \leq n$,  the 
$k$-point correlation function $p^{(k)}_n(x_1,\ldots,x_k)$ is
 defined by the formula 
$$ 
p^{(k)}_n(x_1,\ldots,x_k) :=  \int_{\R^{n-k}} p_n(x_1,\ldots,x_n)\ 
dx_{k+1} \ldots dx_n.
$$

Here we follow the convention that the correlation functions are $L^1$-normalized.
Other conventions are also used in the literature. e.g. the
correlation functions $R_k$ used in \cite{Deibook, Joh} are related to
$p^{(k)}_n$ via $R_k = \frac{n!}{(n-k)!} p^{(k)}_n$. 

The main results of this note are as follows.

\begin{theorem}[Universality of averaged correlation function]\label{main1} 
 Fix $\eps >0$ and $u$ such that
 $-2 < u - \eps < u+\eps < 2$. Let $k \geq 1$ and
 let $f: \R^k \to \R$ be a continuous, compactly supported function,
 and let $H=H_n$ be a Wigner random matrix.  Then the quantity
\begin{equation}\label{episode}
 \frac{1}{2\eps} \int_{u-\eps}^{u+\eps} \int_{\R^k} f(\alpha_1,\ldots,\alpha_k)
 \frac{1}{[\rho_{sc}(u')]^k} p_n^{(k)}\Big(u'+\frac{\alpha_1}{n \rho_{sc}(u')}, 
\ldots, u'+\frac{\alpha_k}{n \rho_{sc}(u')}\Big)\ d\alpha_1 \ldots d\alpha_k du'
 \end{equation}
converges as $n \to \infty$ to
$$ 
\int_{\R^k} f(\alpha_1,\ldots,\alpha_k) 
\det( K(\alpha_i,\alpha_j) )_{i,j=1}^k\ d\alpha_1 \ldots d\alpha_k,
$$
where $K(x,y)$ is the \emph{Dyson sine kernel}
\begin{equation}\label{dyson}
 K(x,y) := \frac{\sin( \pi(x-y) )}{\pi(x-y)}
\end{equation}
and $\rho_{sc}(u)$ is the \emph{Wigner semicircle distribution}
\begin{equation}
 \rho_{sc}(u) := \frac{1}{2\pi} (4-u^2)_+^{1/2}.
\label{sclaw}
\end{equation}
\end{theorem}

\begin{theorem}[Universality of averaged eigenvalue gap]\label{main2} 
Fix  $u$ such that $-2 < u<   2$, choose $k \geq 1$ and $s>0$, 
and let $H$ be a Wigner random matrix.  Then the expectation value of the 
empirical gap distribution
$$ 
  \Lambda_n(u, s, \eps) = 
 \frac{1}{2\eps n \rho_{sc}(u)} \# \Big\{ 1 \leq j < n: \lambda_{j+1} - \lambda_j 
\leq \frac{s}{n \rho_{sc}(u)}; |\lambda_j - u| \leq \eps \Big\}
$$
satisfies
\begin{equation}
\lim_{\eps\to0}
\lim_{n\to\infty} \E \,  \Lambda_n(u, s, \eps) =  \int_0^s \frac{d^2}{d\alpha^2} 
\det(1 - {\mathcal K}_\alpha)\ d\alpha,
\label{limgap}
\end{equation}
where ${\mathcal K}_\alpha$ 
is the integral operator on $L^2((0,\alpha))$ whose kernel is the Dyson 
sine kernel \eqref{dyson}.
\end{theorem}

We now briefly recall the previous partial results towards these theorems. 
 In the case of GUE, the proofs of Theorems \ref{main1}, \ref{main2} can be 
found in \cite{Meh}.  These proofs relied  heavily on the
 explicit formula for the $k$-point correlation functions in the GUE case. 
 An important extension  was made by Johansson \cite{Joh}, who established the above
 theorems in \emph{Gaussian divisible} case when the matrix $H$ took the form  
$\hat H + a V$, where $\hat H$ is another Wigner matrix, $V$ is a GUE matrix independent
 of $\hat H$ and $a>0$ is an arbitrary fixed constant
independent of $n$. This proof also requires the use of an explicit formula, however, its analysis is  much more delicate than in the GUE case. 
By a simple rescaling, 
it is clear that Johansson's result extends to Wigner matrices of the form 
\begin{equation}\label{hoo}
H = e^{-t/2} \hat H + (1-e^{-t})^{1/2} V,
\end{equation}
where $\hat H$ and $V$ are as above, and the time $t>0$ 
 is fixed independent of $n$. One can view $H$ as the evolution
 of $\hat H$ under an Ornstein-Uhlenbeck process of time $t$.

 In \cite{ERSY2},  four of the authors extended Johansson's result to obtain 
 Theorems \ref{main1} and \ref{main2} for
matrices of the form (\ref{hoo}) with a short time $t=n^{-1+\delta}$ 
for any fixed $\delta>0$.
An important ingredient was the convergence of the density of the eigenvalues 
to the semicircle law \eqref{sclaw} on $O(n^{-1+\delta})$ scales 
which was previously established in \cite{ESY3}.
Combining the result for Wigner matrices of the form (\ref{hoo}) 
for time $t=n^{-1+\delta}$ 
with a time-reversal argument, Theorems \ref{main1} and \ref{main2}
were also obtained for Wigner matrices $H$ with no Gaussian component, under the assumption that atom distribution of $H$ has a probability density 
$h (x) = e^{-g(x)}$, where $g \in C^6 (\R)$ grows at 
most polynomially fast at infinity and  $h(x) \leq C e^{-c|x|}$. For more details, we refer to \cite[Section 1]{ERSY2}.

In \cite{TVbulk},  the other two
authors introduced a different way to study the local eigenvalue statistics of random matrices, 
based on a variant of Lindeberg replacement strategy. This enables one to 
compare the statistics of a general Wigner matrix to those of a Johansson matrix, or any random matrix where the desired 
statistics is known, given that some moment condition is satisfied. 
As an application,  Theorems \ref{main1}, \ref{main2} 
were established in \cite{TVbulk} under the additional hypothesis that the atom
 distribution $x$ had vanishing third moment and was supported on at least three
 points.  For more details, we refer to \cite[Section 1]{TVbulk}.

 In this note, we observe that one can combine  the arguments from \cite{ERSY2} with the arguments of 
\cite{TVbulk} to eliminate the regularity hypotheses from the former and the
 moment and support hypotheses from the latter.  The basic idea is simple and can be described as follows. 
 The main result of \cite{TVbulk} (see \cite[Theorem 15]{TVbulk}) 
 allows one to compare the local statistics of two Wigner hermitian matrices whose atom variables $\xi, \xi'$
  have the same third and 
 fourth moments, namely $\E \xi^3 = \E \xi'^3$ and $\E \xi^4 = \E \xi'^4$. 
 A closer look reveals that it is enough  to require that the differences
 $|\E \xi^3 - \E \xi'^3|$ and $|\E \xi^4 - \E \xi'^4|$ are sufficiently small
 (in term of $n$).  This allows us  to compare the original matrix $H$ 
 with a Johansson-type matrix 
\eqref{hoo} with $t = n^{-1+\delta}$ for some small $\delta$, instead of comparing with a Johansson matrix with 
$t= \Theta (1)$ as done in \cite{TVbulk}. On the other hand, arguments from \cite{ERSY2}
guarantee  that such a Johansson matrix has the desired statistics.

One defect of our method is that the statistics need to be averaged over a 
spatial scale $\eps$ which is independent of $n$ before universality is established. 
 This appears to be a largely technical condition, arising from the lack of good
 concentration results for individual eigenvalues $\lambda_i$ for matrices of 
the form \eqref{hoo} (in particular, we do not know how to eliminate the vanishing
 third moment hypothesis from \cite[Theorem 29]{TVbulk}).  However, if we reimpose
 the hypothesis that the third moment  of the atom distribution
vanishes (i.e., $\E \, x^3 =0$), we can remove the averaging in Theorem \ref{main1}, thus refining 
  \cite[Theorem 11]{TVbulk} by removing the condition that the support must have at least 3 points
  (see Section \ref{section:remark}). 
 
 \section{Proof of theorems}

We now prove Theorem \ref{main1}.  We fix $u,\eps,f,H$; by a limiting 
argument we may take $f$ to be smooth.  We will always assume $n$ to be 
sufficiently large depending on these parameters.  Using the exponential 
decay of $x$ and a standard truncation argument (see \cite[Section 3.1]{TVbulk})
 we may assume that
\begin{equation}\label{x-bound}
x = O( \log^{O(1)} n )
\end{equation}
almost surely.  (Note that the Gaussian random variable $N(0,1/2)$ does not quite 
obey this bound, but can be modified on a set of extremely small probability to 
do so whenever necessary, so we may effectively assume that \eqref{x-bound}
 holds in this case also.)

We introduce the modified Wigner matrix
$$ H' := e^{-t/2} H + (1-e^{-t})^{1/2} V,$$
where $t := n^{-1+\delta}$  and $V$ is a GUE matrix independent of $H$. 
 We will choose $\delta=0.01$ (say), but the proof below actually requires
only $\delta<1/2$.  

 Observe that $H'$ is another Wigner random matrix, with atom distribution $x'$ given by
$$ x' := e^{-t/2} x + (1-e^{-t})^{1/2} g$$
where $g \equiv N(0,1/2)$ is a Gaussian random variable independent of $x$. 
 In particular we have the moment comparision bounds
\begin{equation}\label{ex0}
\E (x')^j = \E x^j
\end{equation}
for $j=0,1,2$, and
\begin{equation}\label{exx}
\E (x')^j = \E x^j + O( n^{-1+0.01} )
\end{equation}
for $j=3,4$.

Let $(p'_n)^{(k)}$ be the $k$-point correlation function associated to $H'$.  
The first step is to show that Theorem \ref{main1} is valid for $H'$:

\begin{proposition}\label{prop5} The quantity
$$ \frac{1}{2\eps} \int_{u-\eps}^{u+\eps}  \int_{\R^k} f(\alpha_1,\ldots,\alpha_k)
 \frac{1}{[\rho_{sc}(u')]^k} (p'_n)^{(k)}\Big(u'+\frac{\alpha_1}{n \rho_{sc}(u')},
 \ldots, u'+\frac{\alpha_k}{n \rho_{sc}(u')}\Big)\ d\alpha_1 \ldots d\alpha_k du'$$
converges as $n \to \infty$ to
$$ \int_{\R^k} f(\alpha_1,\ldots,\alpha_k) \det( K(\alpha_i,\alpha_j) )_{i,j=1}^k\ 
d\alpha_1 \ldots d\alpha_k.$$
\end{proposition}

\begin{proof} By the dominated convergence theorem, it suffices to show that
\begin{equation}\label{eq:ufix}
\begin{split}
\lim_{n\to \infty}  \int_{\R^k} f(\alpha_1,\ldots,\alpha_k) &
\frac{1}{[\rho_{sc}(u)]^k} (p'_n)^{(k)} \Big(u+\frac{\alpha_1}{n \rho_{sc}(u)}, 
\ldots, u+\frac{\alpha_k}{n \rho_{sc}(u)}\Big)\ d\alpha_1 \ldots d\alpha_k \\
 &= \int_{\R^k} f(\alpha_1,\ldots,\alpha_k) \det( K(\alpha_i,\alpha_j) )_{i,j=1}^k\
 d\alpha_1 \ldots d\alpha_k. \end{split}\end{equation}
for every fixed $-2 < u < 2$. The dominated convergence theorem
can be applied since the integral on the lhs of \eqref{eq:ufix} can be bounded by 
$C_k\, \E \, N_I^k$, where $N_I$ is the number of eigenvalues in
an interval $I$ of size $C/n$ about $u$. The boundedness of $\E \, N_I^k$
follows from  \cite[Theorem 5.1]{ESY3}. Originally this result was stated
for atom distribution with Gaussian decay, but the proof can be
modified for random variables with subexponential decay
and satisfying \eqref{x-bound}.

Eq. (\ref{eq:ufix}) is proven in Proposition 3.1 of \cite{ERSY2}.
The proof is based on an explicit integral representation for the marginal 
$(p'_n)^{(k)}$. The only assumption used to prove (\ref{eq:ufix}) in \cite{ERSY2}
 is the asymptotic validity of the local semicircle law for the density of the 
eigenvalues of the 
Wigner matrix $H$ (with no Gaussian part) on scales of order $n^{-1+\delta}$.
 More precisely, if
\[ m_n (z) = \frac{1}{n} \tr \frac{1}{H-z} , \quad \mbox{and} \quad
 m_{sc} (z) = \int \frac{ds \, \varrho_{sc} (s)}{s-z}
\] 
denote the Stieltjes transforms
of the empirical eigenvalue distribution and of the semicircle law, respectively,
then we need to know 
\cite[Eq. (3.9)]{ERSY2} 
that for any $\kappa,\delta>0$ there exist constants $c,C>0$
such that
\begin{equation}\label{eq:semi}
 \P \left( \sup_{E \in (-2+\kappa ; 2 - \kappa)} 
\left| m_n (E+i\eta) - m_{sc} (E+i\eta)\right| \geq \eps \right) 
\leq C e^{-c\eps \sqrt{n \eta}} 
\end{equation} 
for $ n^{-1+\delta}\le\eta\le 1$ and
for all $\eps >0$ sufficiently small.

Eq. (\ref{eq:semi}) has been established in \cite{ESY3} under the assumption of 
Gaussian decay of the atom distribution, that is assuming that 
$\E e^{\beta x^2} < \infty$, for some $\beta >0$. 
To reduce this assumption to the desired sub-exponential decay assumption, 
we can  use   a result in \cite{TVbulk}
(see the paragraph following \cite[Lemma 60]{TVbulk})  which proved \eqref{eq:semi} under this weaker assumption with 
a weaker probability bound $\exp(-\omega (\log n))$, where $\omega(\log n)$ denotes a quantity whose ratio with $\log n$ goes to infinity as $n \to \infty$.  One can also follow the ideas in 
 Section 5 of \cite{ERSY2}, where it is shown that the 
Gaussian decay can be relaxed to exponential decay
(at the expense of a slightly weaker probability bound). These ideas can be further 
refined to deal with sub-exponential decay.

Following the proof of Proposition 3.1 in \cite{ERSY2},
it is clear that the argument holds 
using this weaker bound on the probability  as well. This concludes the proof of Proposition 
\ref{prop5}.
\end{proof}

To conclude the proof of Theorem \ref{main1}, it thus suffices to show that 
the expression \eqref{episode} only changes by $o(1)$ when we replace $H$ 
with $H'$, where we use $o(1)$ to denote an expression that goes to zero 
as $n \to \infty$.

From the definition of $p_n^{(k)}$, we can rewrite \eqref{episode} as
\begin{equation}\label{central}
\frac{1}{n} \sum_{1 \leq i_1,\ldots,i_k \leq n \hbox{ distinct}} 
\E  \, G( n \lambda_{i_1}, \ldots, n \lambda_{i_k} )
\end{equation}
where
\begin{equation}\label{G-def}
G(y_1,\ldots,y_k) := \frac{n}{2\eps} \int_{u-\eps}^{u+\eps} 
\frac{d u'}{[\rho_{sc}(u')]^k} f\Big(\rho_{sc}(u') (y_1 - nu'), 
\ldots, \rho_{sc}(u') (y_k - nu')\Big).
\end{equation}
On the other hand, by \cite[Proposition 61]{TVbulk} (or \cite[Theorem 5.1]{ESY3}), 
taking advantage of the hypothesis \eqref{x-bound}, we see that, with probability 
$1-O(n^{-100})$ (say),
$$ \# \{ 1 \leq i \leq n: \lambda_i \in I \} \leq \widetilde{C} n |I|$$
for all intervals $I \subset [u-\eps,u+\eps]$ with $|I| \geq \log^C n / n$, for 
some sufficiently large absolute constant $C$, $\widetilde{C}$, where $\widetilde{C}$ can depend on $u$ and $\eps$.  Because of this, and the compact 
support of $f$, we can restrict the summation in \eqref{central} to those 
$i_1,\ldots,i_k$ for which $i_j = i_1 + O( \log^{O(1)} n )$ for all $1 \leq j \leq k$,
 while only paying an acceptable error of $o(1)$. In particular, there are now only
 $O( n \log^{O(1)} n )$ summands in \eqref{central}.  Similarly, using the convergence
 to the Wigner semicircular law (see e.g. \cite{Pas}) 
 we may restrict $i_1,\ldots,i_k$ to
 the interval $[\delta n, (1-\delta) n]$ for some fixed $\delta > 0$ independent of $n$.

A routine computation shows that $G$ enjoys the derivative bounds
$$ |\nabla^j G(y_1,\ldots,y_k)| = O(1)$$
for all $y_1,\ldots,y_k$ and $0 \leq j \leq 5$, where the implied constant in the
 $O()$ notation can depend on $u, \eps, f$.  

Suppose first that that the random variables $x, x'$ matched to fourth order.  
One could then apply the four moment theorem in \cite[Theorem 15]{TVbulk} to 
conclude for each $i_1,\ldots,i_k \in [\delta n, (1-\delta) n]$, the quantity 
$\E  G( n \lambda_{i_1}, \ldots, n \lambda_{i_k} )$ changes by $O(n^{-c})$ for 
some absolute constant $c>0$, uniformly in $i_1,\ldots,i_k$, when $H$ is replaced
 by $H'$.  Inserting this into \eqref{central}, we obtain the desired claim.

In our situation, $x$ and $x'$ do not match exactly to fourth order, but only to 
second order.  However, we have approximate matching to third and fourth order thanks
 to \eqref{exx}, and this is enough to recover the conclusions of the four moment
 theorem.  Indeed, an inspection of the proof of \cite[Theorem 15]{TVbulk} in 
\cite[Section 3.3]{TVbulk} reveals that the matching moments condition is used 
in only one place, namely to show that the expectation of an expression of the form
$$
 F(0) + \nabla F(0)( h_{\ell k} ) + \ldots + \frac{1}{4!} 
\nabla^4 F(0)( h_{\ell k}, \ldots, h_{\ell k}) + O( n^{-5/2 + o(1)} )
$$
only changes by at most $n^{-2-0.1}$ (say) when one replaces $h_{\ell k} = 
n^{-1/2} ( x_{\ell k} + \sqrt{-1} y_{\ell k} )$ with $h'_{\ell k} := 
n^{-1/2} ( x_{\ell k} + \sqrt{-1} y_{\ell k} )$, where $1 \leq \ell < k \leq n$, 
and $F: \C \to \R$ is a function obeying the derivative bounds
$$ |\nabla^j F(0)| \leq C n^{-j + o(1)}$$
for $0 \leq j \leq 4$.  But the expectation of the first three terms of this
 expansion do not change when replacing $h_{\ell k}$ with $h'_{\ell k}$ thanks to
 \eqref{ex0}, and the next two terms change by $O( n^{-3+o(1)} n^{3/2} n^{-1+0.01} )$ 
and $O( n^{-4+o(1)} n^{4/2} n^{-1+0.01} )$ respectively, and the claim follows.

\begin{remark} The reader will observe that there is some room to spare here regarding powers of $n$.  Indeed, a more careful examination of the argument reveals that one can sharpen Theorem \ref{main1} by not keeping $\eps$ fixed with respect to $n$, but instead letting it decay as $\eps = n^{-1/2 + \delta}$ for any $0 < \delta < 1/2$.  This should not be the optimal result, however; the correct scale should be $O(n^{-1+\delta})$ (and in fact one should not need to average in $u$ at all).  Unfortunately, our methods so far can only achieve this by putting more moment or regularity hypotheses on the distribution $x$, see for instance Section \ref{section:remark}.
\end{remark}

We now sketch the proof of Theorem \ref{main2}.
The first step is to prove that the empirical gap distribution $\Lambda_n'$
of the matrix $H'$ satisfies the equation \eqref{limgap}. 
This follows from the same argument
used in the proof of of Theorem 1.2 in \cite{ERSY2}. 
The only difference is that the parameter
$\eps$ in \eqref{limgap} was replaced by $n^{-1 + \delta}$
 (with some $0<\delta <1$) in \cite{ERSY2}.
Since the change of the density $\varrho_{sc}(u')$ vanishes 
as $\eps \to 0$, the same proof applies to the current
setting, using that  the limiting  function on the rhs of \eqref{limgap}
is continuous in $s$.

In the second step we deduce \eqref{limgap} for $H$ by using
that the expected values of
observables   of the form \eqref{central} are asymptotically
the same for $H$ and $H'$. By  the exclusion-inclusion
principle, $\Lambda_n$ can be written as an alternating series
of expressions in the form \eqref{central} with functions 
$G$ which are products of characteristic
functions of intervals of length one  or larger.
These characteristic functions can be approximated
by  bounded smooth functions  in $L^1$-sense
with arbitrary precision $\kappa$. Controlling the alternating
series by its last term,
letting first $n\to\infty$ then $\kappa\to 0$
and using the continuity in $s$ 
of the limiting function, we obtain \eqref{limgap}
for $H$ as well.

\section{Concluding remarks} \label{section:remark} 

 If one assumes the additional third moment condition
 $\E x^3 = 0$ on the moment condition, then it was shown in \cite[Theorem 29]{TVbulk}
 that the eigenvalues $\lambda_i$ with $i \in [\delta n, (1-\delta) n]$ obeyed the 
localization property
$$ \lambda_i = t(\frac{i}{n}) + O(n^{-1+c})$$
with probability $1-O(n^{-C})$ for any $C, c>0$ (and $n$ sufficiently large 
depending on $C,c,\delta$), where $t(a)$ was defined by the formula
$$ a := \int_{-2}^{t(a)} \rho_{sc}(x)\ dx.$$
Because of this localization, it is possible to eliminate the averaging over
 $u'$ in Theorem \ref{main1}, and to reduce the averaging in Theorem \ref{main2}
 to an interval of $O(n^\delta)$ eigenvalues rather than $O(n)$, by combining the 
above the arguments with those used to prove \cite[Theorems 9,11]{TVbulk}.  The net
 effect of this is to eliminate the hypothesis in those theorems that $x$ is 
supported on at least three points. Thus, for instance, we can take  $x$ to be the Bernoulli 
variable  that takes values  $\pm 1$ with an equal probability $1/2$.

If the third moment condition from \cite[Theorem 29]{TVbulk} could be omitted, 
then one could similarly reduce the averaging required for Theorems \ref{main1}, 
\ref{main2} even when the third moment was non-zero. To do so, however, would 
require either a relaxation of the moment conditions needed  in \cite[Theorem 15]{TVbulk}, or else an analogue of \cite[Theorem 29]{TVbulk}
 for the matrices $H'$ studied here.  It seems plausible that at least one of these 
two approaches could eventually be made to work, but this does not appear to be 
easily deducible from the existing literature.

Another remark is that our theorems hold for a more general model of random matrices, where the real and imaginary parts of the entries are not necessarily independent. The reason is that \cite[Theorem 15]{TVbulk} (and its variant used here) do not require this assumption 
(see also the end  of \cite[Section 1]{TVbulk}).

\end{document}